\documentclass[10pt]{amsart}
\usepackage{amsmath}
\usepackage{amssymb}
\usepackage{xcolor}
\usepackage{enumitem} 
\usepackage{tikz-cd}
\usepackage[all]{xy}
\newtheorem{theorem}{Theorem}[section]
\newtheorem{corollary}[theorem]{Corollary}
\newtheorem{proposition}[theorem]{Proposition}
\newtheorem{definition}[theorem]{Definition}

\newtheorem{remark}[theorem]{Remark}
\newtheorem{lemma}{Lemma}

\newcounter{question}

\title{Polynomials as Lipschitz maps on the Veronese cone}

\author{Maite Fernández-Unzueta}

\address{
CIMAT, Calle Jalisco S/N, Mineral de Valenciana \\
Guanajuato, Gto. C.P. 36023  \\ M\'exico}
\email{maite@cimat.mx}
\date{\today}


\begin{document}

\thispagestyle{empty}

\begin{abstract}
 Given a Banach space $X$ and $d\in \mathbb{N}$,  we construct a metric space $\mathbb{V}_X^d$  with the property that every $d$-homogeneous polynomial defined on $X$ factors through a Lipschitz map  on it. We prove that the metric on  $\mathbb{V}_X^d$ is independent (up to a constant) of the  norm of the tensor space  in which it is embedded. We apply this fact to prove that a homogeneous polynomial is  Lipschitz $q$-summing  as a polynomial if and only if its associated Lipschitz map is   Lipschitz $q$-summing. This result generalizes the already known theorem for linear operators

\end{abstract}

\maketitle

\section{Introduction}

The non-linear theory of Banach spaces develops in multiple directions. In this paper, we will focus on two of them: one is the generalization of the linear theory of bounded operators to that of polynomial maps, and the other is the generalization of the linear theory to Lipschitz maps.
It is worth noting that no homogeneous polynomial between Banach spaces is a Lipschitz map, except if it is linear. This could be one of the reasons why the two generalizations that we just mentioned  have followed separate paths.

Nevertheless, in this paper we establish a bridge between both theories, unambiguously associating  a Lipschitz map to each homogeneous polynomial, preserving throughout the intrinsic information of the polynomial. This construction allows  to incorporate  Lipschitz theory into  the study of polynomials. This  is
detailed in  Theorem \ref{thm: in symmetric terms} and Proposition \ref{prop: V-operators in Lipschitz}.
It is expected that this construction will be applicable to  different aspects of the  study of polynomial maps between Banach spaces. Here, we present an application  which concerns the property of $q$-summability of operators.

To achieve this goal — the construction of a nexus between both theories — we  first construct the metric space $\mathbb{V}^d_X$, which will be the domain of these Lipschitz transformations.
This metric space, which  we  call the {\it Veronese cone},    will depend only on the degree of the homogeneous  polynomial and its  domain  $X$ and will be independent (up to a constant) of the reasonable  cross-norm required to define it.
This fact, proved in Theorem \ref{thm: metrics concide}, will allow us to use these type of   norms,  as needed. Such is the case in Theorem \ref{thm: Lip=Lip}, where the proof required the use of two reasonable cross-norms.

 As an application of the relation we have just established,
 we will address the comparison of the notions of $q$-summability in the  two previous   non-linear contexts. The  property for bounded linear operator was introduced by A. Grothendieck in \cite{Groth56} as {\it semi-integral \`a droit}   when  $q=1$ and by  A Pietsch \cite{Pietsch} when $1<q<\infty$.
 A fundamental fact about the operators in this  class is that they are exactly those that linearly  factor by  a restriction of the natural inclusion  $i_{\infty,q}: L_{\infty}(\mu)\hookrightarrow L_{q}(\mu)$ for some probability measure \cite[Theorem 2.13]{DJT}.
 The definition of $q$-summability for  Lipschitz mappings between metric spaces was introduced in \cite{FJ09} by J. Farmer and W.B. Johnson. There, the authors  proved that a Lipschitz map belongs to this class if and only if it  admits a factorization by Lipschitz maps, one of which is   the restriction
 of   the natural inclusion map $i_{\infty, q}$  for some probability measure. The study  of this property for polynomials has a more involved history, as many different notions have been proposed (see e.g. \cite{Dimant03}, \cite{Matos03}, \cite{Pietsch}, or \cite{PelRueSaP}, where several of them can be found).  In this paper, we will refer to the $q$-summability property for polynomials introduced in \cite{AngFU} by J.C. Angulo and the author, as it is the  one (besides the trivial one, where the polynomial property is defined as the linear property of the  associated linear opearator) that admits a characterization through factorization  by the restriction of  the inclusion  $i_{\infty, q}$, along with various other natural generalizations of key properties of the linear class.

 We use the  geometric tool already developed, to  prove that for a homogeneous polynomial,  being $q$-summable as a polynomial (in the sense introduced in \cite{AngFU}) is equivalent to being $q$-summable as a Lipschitz function.
This result generalizes the theorem proved in \cite{FJ09} for bounded linear operators.

The results proven here are part of a more general development, whose main objective is to leverage the algebraic structure of the involved operators, primarily defined on infinite-dimensional spaces. The case of multilinear operators was treated in \cite{SegreCone20}, where it was proved that a bounded multilinear mapping factors uniquely through a Lipschitz mapping defined on the metric space called the Segre cone. The definitions of the (finite-dimensional) Segre variety and  of the Veronese variety can be found in Shafarevich´s book \cite{Shafarevich}.

We refer the reader to \cite{Dineen99} and \cite{HajJoh} for general background on polynomial mappings;   to \cite{BenyLind} and \cite{CMN LipF book} for Lipschitz mappings; to  \cite{DefFlo93},  \cite{DieFouSwa}  and \cite{Ryan-libro} for tensor products on Banach spaces and to \cite{DJT} and \cite{PisierCBS60} to
absoluetly  summing operators.

\section{ The  Veronese cone of a Banach space }

Let $X, Y $  be a Banach space  (real or complex) and let $d\in  \mathbb{N}$.  A  map $P: X\rightarrow Y$  is a
{\sl $d$-homogeneous polynomial} if there exists  a continuous $d$-linear map $T: X\times
\buildrel d\over \cdots \times X \rightarrow Y $ such that $P(x)=T(x,\ldots,x)$ for every $x\in X$.
 The multilinear map is unique if it is required to be symmetric.  We denote $  \otimes^d X $ or $ X\otimes  \stackrel{d}{\cdots}\otimes X$    the  tensor product   of degree $d$ of the space $X$, as a vector space.  A  norm defined on  $ \otimes^d X$ is
 a {\it reasonable cross norm} if it has the following properties:
  \begin{enumerate}[label=(\roman*)]
  \item $\alpha(x_1\otimes\dots\otimes x_d)\leq \|x_1\|\cdots \|x_d\|$ for every $x_i\in X;\; i=1,\ldots d. $
  \item For every $x_i^*\in X^*$, the linear functional $x_1^*\otimes\dots\otimes x_d^*$ on $\hat{\otimes}_{\alpha}X $ is bounded, and $\|x_1^* \otimes\dots\otimes x_d^*\|\leq \|x_1^*\|\cdots \|x_d^*\|.$
  \end{enumerate}

A norm $\alpha$ on the vector space $\otimes^d X  $ is a reasonable cross norm if and only if for  every  $ z\in {X}\otimes \cdots \otimes X$
\begin{equation}\label{eq: inject proj tensor}
\epsilon(z)\leq \alpha(z)\leq \pi(z)\end{equation}
 where $\epsilon$ and $\pi$ are, respectively,  the injective tensor norm and the projective tensor norm on   ${X}\otimes \cdots \otimes X$  (for the original proof, see \cite[Theorem 1 p.8]{Groth56}; see also  \cite[Theorem 1.1.3]{DieFouSwa}, \cite[Proposition 6.1]{Ryan-libro}).   Inequalities (i) and (ii) are, in fact,   equalities.
The  normed space determined by a reasonable cross norm  $\alpha$ on the space  $\otimes^d X$  will be  denoted by $\otimes^d_{\alpha}X $ and its completition by    $ \hat{\otimes}^d_{\alpha}X$.

We will say that  $\alpha$ is a {\sl tensor norm}  if for every  $d$-tuple of Banach spaces
$\{X_i\}_{i=1}^d $  $\alpha$ is a reasonable cross-norm on the algebraic tensor product $X_1\otimes\cdots \otimes X_d$  and if it satisfies the so called {\sl metric mapping property}, namely, that   whenever   $T_i\in \mathcal{L}(X_i, Y_i)$, and $T_1\otimes\cdots \otimes  T_d$ is the linear mapping determined by the relation $(T_1\otimes\cdots \otimes  T_d)(x_1\otimes\cdots \otimes x_d):=T_1(x_1)\otimes\cdots \otimes  T_d(x_d)$, then $T_1\otimes\cdots \otimes  T_d\in \mathcal{L}(X_1\hat{\otimes}_{\alpha}\cdots \hat{\otimes}_{\alpha} X_d , Y_1\hat{\otimes}_{\alpha}\cdots \hat{\otimes}_{\alpha} Y_d )$  and
\begin{equation}\label{eq: tensor norm}  \|T_1\otimes\cdots \otimes  T_d\|  \leq \|T_1\|\cdots\|T_d\|.
\end{equation}

To fix notation and for the sake of completeness  we include  the following  result:

\begin{lemma}\label{lem: veronese map is continuous}
If $\alpha$ is a reasonable cross-norm, then the   mapping $\nu_X^d: X\rightarrow  \hat{\otimes}^d_{\alpha}X$ defined as $\nu_X^d(x):=x\otimes\stackrel{d}{\ldots}\otimes x$ is a continuous $d$-homogeneous polynomial.
\end{lemma}
\begin{proof}
For each $x\in X$ we have that    $\nu_X^d(x)$ coincides with $T(x,\ldots, x)$ where
 $T: X\times\stackrel{d}{\ldots}\times X\rightarrow  \hat{\otimes}^d_{\alpha}X$ is defined as $T(x_1,\ldots,x_d):=x_1\otimes\cdots\otimes x_d,$. This is a multilinear mapping  which is bounded because of   property $(i)$ of  a  reasonable cross norm.

\end{proof}

\begin{lemma}\label{lem: separacion de puntos}

Let $X$ be a complex Banach space and $d\in \mathbb{N}$. Then,
$x\otimes \cdots \otimes x= y\otimes \cdots \otimes y \in \hat{\otimes}^d_{\alpha}X$ if and only if there exists $k=0\ldots d-1$ such that $x=\exp^{\frac{2\pi i k}{d}}y$.   If $X$ is a real Banach space then $x\otimes \cdots \otimes x= y\otimes \cdots \otimes y$  if and only if $x=y$ whenever $d$ is odd and $x=\pm y$ whenever  $d$ is even.
\end{lemma}
\begin{proof} By  property $(ii)$ of a reasonable cross norm,  for every $x^*\in X^*$  the map    $(x^*)^d$ defines an element in the dual space $(\hat{\otimes}^d_{\alpha}X)^*$. If we   assume that $x\otimes \cdots \otimes x= y\otimes \cdots \otimes y$,  then, for every  $x^*\in X^*$, $(x^*)^d(x)=(x^*)^d(y)$.  If $x$ and $y$ are linearly independent, there is some $x^*\in X^*$ such that $x^*(x)=0$ and $x^*(y)\neq 0$. Then $ (x^*)^d(x\otimes \cdots \otimes x)=0$ while $ (x^*)^d(y\otimes \cdots \otimes y)\neq 0$, contradicting the assumption. Consider now the case where $x=\lambda y $  for some $\lambda \in \mathbb{K}$. This implies that $x\otimes \cdots \otimes x= \lambda^d y\otimes \cdots \otimes y$. But this is only possible if $\lambda^d=1$. Both statments, the real and the complex ones, follow from this.
The converse  implication follows by homogeneity.
\end{proof}

The following lemmas will play an important role throughout the text. The case of the projective tensor product  $\alpha=\pi $ can be implicitly found  in  \cite{BombalFU}.
\begin{lemma}\label{lem: tesis-ult al rescate}
Let $\{x_n\otimes\stackrel{d}{\ldots} \otimes x_n\}_n \subset \hat{\otimes}^d_{\alpha}X$  be a  sequence converging to some $z\in\hat{\otimes}^d_{\alpha}X$. Then   there exists  $x\in X$  and a  subsequence $\{x_{n_j}\}_{j}$ such that such that $x_{n_j}\to_j  x$ in $X$ and  $z=x\otimes\cdots\otimes x $.
\end{lemma}
\begin{proof}

Let   $\{x_n\otimes\stackrel{d}{\ldots} \otimes x_n\}_n$  be a sequence that converges to $z$ in $  \hat{\otimes}^d_{\alpha}X$.  In particular it is a bounded sequence. Since $\|x_n\otimes\stackrel{d}{\ldots} \otimes x_n\|=\|x_n\|^d$, we have that $\{x_n\}_n$ is a bounded sequence in $X$.
Even more, from this relation we obtain that if $z=0$, then   $x_{n}\to_n  0$ and the statement is proved in this case.

Assume now that $z\neq 0$. We will prove the statement in two different cases: when $\{x_{n}\}_n$  does not converge weakly to zero and when it does. If the sequence is not weakly convergent  to zero, there exists $x^*\in X^*$ and a subsequence $\{x_{n_j}\}_{j}$ such that
$x^*(x_{n_j})\to_{j} \lambda$ for some $\lambda \in \mathbb{K}\setminus 0.$  The   bounded linear operator $\Psi: \hat{\otimes}^d_{\alpha}X \rightarrow X$ defined as
$\Psi(x_1\otimes \cdots \otimes  x_d):=x^*(x_1)\cdots x^*(x_{d-1})x_d$ will transform the convergent sequence $\{x_{n_j}\otimes\stackrel{d}{\ldots} \otimes x_{n_j}\}_j$ into a convergent sequence $\Psi(x_{n_j}\otimes \cdots \otimes  x_{n_j})=(x^*(x_{n_j}))^{d-1}x_{n_j}\to_j \lambda^{d-1}x$  for some $x\in X$.   Consequently $x_{n_j}\to_j x$ in norm. The remaining case is the case where $\{x_{n}\}_{n}$ converges weakly to zero. Since we have assumed that $z\neq 0$,  there exists a subsequence and   some $\rho >0$ such that
$\|x_{n_j}\|^d> \rho$.  The relation (\ref{eq: inject proj tensor}) between any reasonable cross norm $\alpha$ and the injective tensor norm $\epsilon$ guarantees the continuity of the   formal inclusion $\hat{\otimes}^d_{\alpha}X \subset \hat{\otimes}^d_{\epsilon}X$. Then  $\{x_n\otimes\stackrel{d}{\ldots} \otimes x_n\}_n$  is also a norm convergent sequence in  $\hat{\otimes}^d_{\epsilon}X$. Since $\{x_{n}\}_{n}$ converges weakly to zero in $X$, then $\{x_n\otimes\stackrel{d}{\ldots} \otimes x_n\}_n$ is weakly convergent to $0$ in $\hat{\otimes}^d_{\epsilon}X$.
Whenever the norm and the weak limit exist, they must coincide. But this contradicts the fact that $\|x_{n_j}\|^d> \rho$.

\end{proof}

Let us denote
$ \mathbb{V}^d_X:=\{x\otimes \stackrel{d}{\cdots}\otimes x \in  \otimes^d X; \, x\in X\}$
 and let    $\Sigma_{X,\ldots,X}:= \left\{ x_1\otimes \cdots\otimes x_d,  x_i\in X\right\}$. They are often  referred to as  {\sl diagonal sets}. If $d_\alpha$ denotes  the distance  induced  by the norm of $\hat{\otimes}^d_{\alpha}X$, it holds that the metric space  $(\Sigma_{X,\ldots,X},    d_{\alpha})$   is complete \cite[Theorem 2.1 and Proposition 2.12]{SegreCone20}. Then, it holds
\begin{corollary}\label{cor: metric subspace}
    The metric space  $(\mathbb{V}^d_X,d_\alpha)$ is a closed subspace of   $(\Sigma_{X,\ldots,X},    d_{\alpha})$. Consequently, it is closed in $\hat{\otimes}^d_{\alpha}X$.
\end{corollary}

\begin{theorem}\label{thm: metrics concide}
Let  $X$ be a Banach  space, $d \in \mathbb{N}$ and    $\alpha$,    $\beta$ reasonable cross-norms defined  on  $\otimes^d X$.  Then,
\begin{enumerate}[label=(\roman*)]

    \item The   metric spaces    $(\mathbb{V}^d_X, d_{\alpha}) $ and $  (\mathbb{V}^d_X, d_{\beta}) $  are  bi-Lipschitz equivalent through the identity mapping.
   It holds that  for every $w,z \in \mathbb{V}^d_X$
  $\;  d_{\alpha}(w,z)\leq 2^{d-1}  d_{\beta}(w,z)$.
 \item  $(\mathbb{V}^d_X, d_{\alpha}) $ is  a complete metric space.
 \end{enumerate}
\end{theorem}
\begin{proof}
In the proof we will mainly use
 \cite[Theorem 2.1]{SegreCone20}. It is proved there     that for every pair of reasonable cross norms $\alpha, \beta$, the identity mapping between the  metric spaces $(\left\{x_1\otimes \cdots\otimes x_d, \,x_i\in X\right\};  d_{\alpha})$ and $(\left\{x_1\otimes \cdots\otimes x_d, \,x_i\in X\right\};  d_{\beta})$ is  bi-Lipschitz with constants bounded by $2^{d-1}$.
Consequently,  the identity mapping restricted to    $\mathbb{V}^d_X$ is  bi-Lipschitz, with constant $2^{d-1}$.   So, $(i)$ is proved.   To prove $(ii)$ consider the  case where $\alpha$ is the   the projective tensor norm  $\pi$.    Let $\{x_n\otimes\cdots \otimes x_n\}_n$ be a Cauchy sequence in $(\mathbb{V}^d_X, d_{\pi})$.   Then, it is also a Cauchy sequence in the space   $(\left\{x_1\otimes \cdots\otimes x_d, \,x_i\in X\right\};  d_{\pi}):=\Sigma_{X,\ldots,X}$, the   so called {\sl Segre cone} of the spaces $X_1=X,\ldots, X_d=X$.     By \cite[Proposition  2.12]{SegreCone20}, it holds that $(\Sigma_{X,\ldots,X}, d_\pi)$ is a complete space.Thus, the sequence converges to a limit in  $\Sigma_{X,\ldots,X}$ which,  by Corollary \ref{cor: metric subspace} must be in  $\mathbb{V}^d_X$. The case of an arbitrary tensor norm $\alpha$ follows from the already proved statement $(i)$.
 \end{proof}

\begin{proposition}\label{prop: subspaces}  Let    $Z$ be  a closed (linear) subspace of a Banach space $X$, $d\in \mathbb{N}$ and $\alpha$  a tensor norm.
\begin{enumerate}
    \item  $(\mathbb{V}^{d}_{Z},d_{\alpha}) $  is a closed metric subspace of  $(\mathbb{V}^{d}_{X},d_{\alpha}) $. The inclusion mapping is bi-Lipschitz with constant $2^{d-1}$.  If $\alpha$ is the injective tensor norm $\epsilon $, then it is an isometric embbeding.
    \item If $Z$ is an complemented subspace of $X$, then $(\mathbb{V}^{d}_{Z},d_{\alpha}) $ is the image of $(\mathbb{V}^{d}_{X},d_{\alpha}) $
by a bounded linear projection defined on $\hat{\otimes}^d_{\alpha}{X}$.
\end{enumerate}
\end{proposition}
\begin{proof}
To prove $(1)$, let $J$ be the inclusion mapping $J:Z \hookrightarrow X$.
Then,   ${J}^{\otimes d}\in\mathcal{L}(\hat{\otimes}^d_{\alpha}{Z},\hat{\otimes}^d_{\alpha}{X})$ and $\|J^{\otimes d}\|\leq 1. $ It is immediate to  see that  $J^{\otimes d}(\mathbb{V}^{d}_{Z})\subset \mathbb{V}^{d}_{X}. $ Now  we check that  the map ${{J}^{\otimes d}}$ restricted  to $ {\mathbb{V}^{d}_{Z}}$ is injective. Let $z_1, z_2 \in Z$, then
\begin{align*}
{J}^{\otimes d}(z_1\otimes \cdots \otimes z_1)={J}^{\otimes d}(z_2\otimes \cdots \otimes z_2) \hspace{.5cm} \Longleftrightarrow  \hspace{.5cm}
J(z_1)\otimes \cdots \otimes J(z_1)= J(z_2)\otimes \cdots \otimes J(z_2)
\end{align*}
By  Lemma \ref{lem: separacion de puntos}, this implies that $J(z_1)=\lambda J(z_2) $ with $\lambda^d=1$. Since $J$ is injective,  it must be $z_1=\lambda z_2$ and consequently $z_1\otimes \cdots \otimes z_1=z_2\otimes \cdots \otimes z_2$.  Using $\|J^{\otimes d}\|\leq 1$, we see that the restriction of $J^{\otimes d}$ to $\mathbb{V}^{d}_{Z} $ is a $1$-Lipschitz mapping.
To prove that   that the inverse map defined on the range  is Lipschitz,  we consider the injective tensor norm $\epsilon$. The injectivity property of this norm  establishes that
${J}^{\otimes d}: \hat{\otimes}^d_{\epsilon}{Z} \rightarrow\hat{\otimes}^d_{\epsilon}{X}$ is an isometry onto its image \cite[4.3 Proposition]{DefFlo93}.  Consequently, the inverse mapping $({J}^{\otimes d})^{-1}$ is well defined in ${J}^{\otimes d} (\hat{\otimes}^d_{\epsilon}{Z})= \hat{\otimes}^d_{\epsilon}{J}({Z})$ and satisfies $({J}^{\otimes d})^{-1}=({J}^{-1})^{\otimes d}$ in this domain. Then,
the inclusion $(\mathbb{V}^{d}_{Z},d_{\epsilon}) \hookrightarrow (\mathbb{V}^{d}_{X},d_{\epsilon})$ is an isometric embedding.
The case of an arbitrary tensor norm $\alpha$ follows by using the  bi-Lipschitz equivalences proved in Theorem \ref{thm: metrics concide}.

To prove $(2)$, let $\Pi\in \mathcal{L}(X,X)$ be a bounded linear  projection onto $Z$. By the uniform property  (\ref{eq: tensor norm}) of the tensor norm $\alpha$, ${{\Pi}}^{\otimes d}$ is a well defined projection  on $\hat{\otimes}^d_{\alpha}{X}$ with  image
$\hat{\otimes}^d_{\alpha}Z$. Since   ${{\Pi}^{\otimes}}(x\otimes \cdots \otimes x)=\Pi(x)\otimes \cdots \otimes \Pi(x) \in \mathbb{V}^{d}_{Z}$,  we see by  evaluating on     every   $x\in Z $, that the image of this mapping is the whole space  $\mathbb{V}^{d}_{Z}$.
\end{proof}

 To work   with polynomials, it is often convenient to use  symmetric tensors.   Let   $\otimes_s^dX$ denote the algebraic symmetric tensor product of degree $d$.
The function
\begin{equation}\label{eq: symm tensor norm}
    \|u\|_{s, \pi}:=\inf \left\{ \sum_i \|x_i\|^d; \; x_i\in X \; \mbox{where } \; u=\sum_i x_i\otimes\stackrel{d}{\cdots}\otimes x_i \right\}
 \end{equation}
is a norm on $\otimes_s^dX$. It   is equivalent to the restriction of the projective tensor norm to the subspace  $\otimes_s^dX$ and satisfies that  $P(^dX)$ and $(\hat{\otimes}_{s,\pi}^dX)^*$ are isometric spaces
\cite[Corollary 2.1]{RyanThesis}.    From now on, we will consider the norm $\|\cdot\|_{s, \pi} $ on the space of  symmetric tensors  and we will denote the completition of the space with respect to this norm as
$\hat{\otimes}^d_{s,\pi}X$.
By  Theorem \ref{thm: metrics concide},   the {\sl diagonal set}  has an intrinsic metric structure, which  depends  only on the metric of  $X$ (up to a constant). Thanks to this, we can define:
\begin{definition}\label{def: hom Veron} Let   $X$ be a Banach space and $d\in \mathbb{N}$.   We will say that
 $(\mathbb{V}^d_{X}, d_{s,\pi})$ is the {\sl Veronese cone} {\sl of degree $d$} of the Banach space $X$. \end{definition}

\section{Factorization of   polynomials  through the Veronese cone}

We shall denote $\mathcal{P}( ^dX,Y)$
the space of $d$-homogeneous polynomials   and     $\mathcal{P}( ^dX)$ when $Y$  is the scalar field. They are Banach spaces with the norm $\|P\|=\sup\{\|P(x)\|;\, x\in X, \|x\| \leq 1\}$.     Given a homogeneous polynomial $P\in \mathcal{P}(^d X,Y)$, there is a unique bounded linear operator $T_{{P}}\in \mathcal{L}(\hat{\otimes}_{s,\pi}^d {X},Y)$ satisfying
${P}(x)=T_{{P}}(x\otimes\stackrel{d}{\ldots}\otimes x)$ for every  $x\in{X}$. This relation defines an isomorphism   between the spaces of mappings which is  isometric  when considering  the  norm defined in (\ref{eq: symm tensor norm}),

\begin{equation}\label{eq: isom isometry}
  \mathcal{P}(^d X,Y) \equiv \mathcal{L}(\hat{\otimes}^d_{s,\pi}X, Y)   \end{equation}

In the sequel, we will say that a polynomial  and a linear operator are associated with each other if they correspond under the isometry in (4).

\begin{theorem}\label{thm: in symmetric terms}
      Let  $P\in \mathcal{P}(^{d}{X},Y)$ and let  $T_{{P}}\in \mathcal{L}(\hat{\otimes}_{s,\pi}^d{X},Y)$ be the  linear operator associated with ${P}$. Then   ${T_{{P}}}_{|_{\mathbb{V}^d_{X}}}:\mathbb{V}^d_{X}\rightarrow Y$  is  a Lipschitz mapping
       and  $P$ factorizes as
       \begin{equation*}\label{eq: factorization}
  \begin{tikzcd}
    X\arrow{rr}{P} \arrow[swap]{dr}{\nu^d_{X}} & & Y \\
     & \mathbb{V}^d_{X}  \arrow[swap]{ur}{{T_{{P}}}_{|}} &
  \end{tikzcd}
\end{equation*}
 It holds that  $\|{T_{{P}}}_{|_{\mathbb{V}^d_{X}}}\|_{Lip}=  \|{P}\| $.
If  $\alpha$ is any reasonable cross-norm, then  $\|P\|\leq \|{T_{{P}}}_{|_{(\mathbb{V}^d_{ X},d_\alpha)}}\|_{Lip} \leq \frac{2^{d-1}d^d}{d!} \|{P}\| $
     \end{theorem}
\begin{proof}
 Consider   $P \in  {P}(^{ d}{X},Y)$ and its  associated  operator   $T_{{P}}\in \mathcal{L}(\hat{\otimes}_{s,\pi}^d{X},Y)$.
  Each mapping  in the diagram is well-defined
and  the diagram commutes. Let $u:=x\otimes\cdots\otimes x,\, v:=y\otimes\cdots\otimes y\in \mathbb{V}^d_{X}$.
 \begin{align*}
   \|{T_{{P}}}_{|_{\mathbb{V}^d_{X}}}(u)-{T_{{P}}}_{|_{\mathbb{V}^d_{X}}}(v)\|=\|T_{{P}}(u)-  T_{{P}}(v)\|\leq \|T_{{P}}\|\cdot \|u-v\|_{s,\pi}=\|T_{{P}}\|\cdot d_{s,\pi}(u,  v)=\|{P}\|\cdot d_{s,\pi}(u,  v).
 \end{align*}

 Then, $  \|{T_{{P}}}_{|_{\mathbb{V}^d_{X}}}\|_{Lip}\leq  \|{P}\|$.   By other hand,  the closed unit ball of
 the space $\hat{\otimes}_{s,\pi}^d{X}$ coincides with  the balanced convex hull of the set $\{x \otimes \cdots \otimes x,  x\in {X}, \|x\|\leq 1\} $ \cite[Proposition 2.2]{RyanThesis}. Using that
 \begin{align*}
 \|{P}\|&=\|{T_{P}}\|= \sup \left\{ \|T_{{P}}(v)\|;\; {{v\in \hat{\otimes}_{s,\pi}^d{X}}},\; \|v\|_{s, \pi}\leq 1;  \right\}
 =\sup \left\{ \|T_{{P}}(x\otimes \cdots \otimes x)\|; \,  \|x\|\leq 1  \right\}
 \\
&\leq \sup \left\{ \|T_{{P}}(x\otimes \cdots \otimes x)- T_{{P}}(y\otimes \cdots \otimes y)\|; \; d_{s,\pi}(x\otimes \cdots \otimes x, y\otimes \cdots \otimes y)\leq 1 \| \right\} \leq   \|{T_{{P}}}_{|_{\mathbb{V}^d_{X}}}\|_{Lip}
 \end{align*}
we obtain the equality.  Now consider the case where $T_P$ is continuous with respect to the $\alpha$-norm and  $\hat{\otimes}_{s,\alpha}^d{X}$ denotes the closure with respect to $\alpha$ of the symmetric tensors.  By the result we just proved, along with     Theorem \ref{thm: metrics concide} and  \cite[Proposition 2.2]{RyanThesis}
\begin{multline*}
 \|P\|=\sup_{u\neq v}\frac{\|T_P(u)-T_P(v)\|}{\|u-v\|_{s,\pi}} \leq  \sup_{u\neq v}\frac{\|T_P(u)-T_P(v)\|}{\|u-v\|_{\pi}}  \leq  \sup_{u\neq v}\frac{\|T_P(u)-T_P(v)\|}{\|u-v\|_{d_{\alpha}}} \\
  =  {\|{T_{{P}}}_{|_{(\mathbb{V}^d_{ X},d_\alpha)}}\|_{Lip}}\leq \sup_{u\neq v}\frac{2^{d-1}\|T_P(u)-T_P(v)\|}{\|u-v\|_{\pi}} \leq \frac{2^{d-1}d^d}{d!}\sup_{u\neq v}\frac{\|T_P(u)-T_P(v)\|}{\|u-v\|_{s,\pi}}=\frac{2^{d-1}d^d}{d!}\|P\|.
\end{multline*}
\end{proof}

Note that  even if   $T_{{P}}$ is not continuous with respect to the $\alpha$-norm on  ${\otimes}_{s}^d{X}$, the mapping  ${T_{{P}}}_{|_{(\mathbb{V}^d_{ X},d_\alpha)}}$ is well defined and Lipschitz.

  We will say that a  mapping $f: \mathbb{V}^d_{X} \rightarrow Y $ is a {\sl $\mathbb{V}$}{\sl -operator} if there exist $ T\in \mathcal{L}(\hat{\otimes}_{s,\pi}^d{X},Y)$  such that $f={T}_{|_{\mathbb{V}^d_{X} }}$. Defining   $\mathcal{L}(\mathbb{V}^d_{X} ; Y):=\{f={T}_{|_{\mathbb{V}^d_{X} }}; \; f \, \mbox{is a}\, \mathbb{V}\mbox{-operator} \}$ with the Lipschitz norm, it holds that the space
 $(\mathcal{L}(\mathbb{V}^d_{X} ; Y),\|\cdot\|_{Lip})$ is  isometric   to $
 P(^{ d}X,Y)$.  Indeed, it holds that

 \begin{proposition}\label{prop: V-operators in Lipschitz} With the notation in Theorem \ref{thm: in symmetric terms}
\begin{enumerate}
\item Let $\Psi: {P}(^{d}{X},Y) \rightarrow \mathcal{L}(\mathbb{V}^d_{X} ; Y)$ be defined as  $\Psi(P): ={T_{{P}}}_{|_{\mathbb{V}^d_{X}}}$. Then $\Psi$  is an isometric isomorphism. It is  an isomorphic  isomorphism in the case    $\mathcal{L}((\mathbb{V}^d_{{X}}, d_{\alpha}) ; Y)$
\item  The inclusion map $\varphi :  \mathcal{L}((\mathbb{V}^d_{{X}}, d_{\alpha}), Y)  \hookrightarrow  Lip_0((\mathbb{V}^d_{{X}}, d_{\alpha}),Y) $ is an isometric (linear) inclusion
\end{enumerate}
 \end{proposition}

  \begin{remark} \label{rmk: facotriz linear case} In the case of  $1$-homogeneous polynomials,   the factorization in Theorem \ref{thm: in symmetric terms}    is the  trivial one: in this case $\mathbb{V}^1_{X}=X$, $\nu_X^1=Id$,  ${P}(^{1}{X},Y)=\mathcal{L}(X, Y)$,   and $P={T_{{P}}}={T_{{P}}}_{|_{\mathbb{V}^d_{X}}}\in{P}(^{1}{X},Y). $
  \end{remark}

\section{Lipschitz $q$-summability  for polynomials}

From now on, the Banach spaces $X$ and $Y$ will be   real spaces. In  \cite{FJ09},  it is introduced the   Lipschitz $q$-summing  norm $\pi_q^L(f)$ ($1\leq q <\infty$)  of a (possibly nonlinear) mapping  between  metric spaces $M, N$
 as    the smallest constant $C$ so that for all $(u_i)_i$, $(v_i)_i$ in $M$
 \begin{equation}\label{eq: FarJ}
 \sum d_N(f(u_i), f(v_i))^q\leq C^q \sup_{h\in B_{{M}^{\#}}} \sum |h(u_i)-h(v_i)|^q,
  \end{equation}
 where $ B_{{M}^{\#}}$ is the unit ball of the Lipschitz dual of $M$. Thus, ${M}^{\#}$
is the space
of all real valued Lipschitz functions  defined on $M$ that are zero at $0$, where $0 \in M$ is a fixed point in $M$ and $d_M$ denotes the metric.  A map $f$ is said to be  Lipschitz $q$-summing when  such $\pi_q^L(f)$ exists.

By other hand, various definitions have appeared in the literature that  generalize the notion of  absolutely $q$-summing  linear operators for homogeneous polynomials.   In what follows, we will refer to the following one, which was introduced  in \cite[Definition 7.3]{AngFU}. This class of polynomials   can be characterized by   a domination condition   and by a factorization diagram   \cite[Theorem 7.2]{AngFU}, which generalize the corresponding linear results (for the linear statments, see e.g. \cite[Theorem 2.12 and Theorem 2.13]{DJT}).

A $d$-homogeneous polynomial $P\in\mathcal{P}(^d X, Y)$   between
Banach spaces $X,Y$   is  a {\sl Lipschitz $q$-summing polynomial},  $1\leq q < \infty$,   if there exists   $c>0$ such that for $k\in\mathbb{N}$, $i=1,\ldots,k$ and  every  $x_i, y_i \in X$,
 \begin{equation} \label{eq: def pol -p-sumante}
  \sum_{i=1}^k \left\|P\left(x_i\right)-P\left(y_i\right)\right\|^q
  \leq    c^q\cdot \sup_{p \in B_{\mathcal{P}(^{ d} {X})}} \sum_{i=1}^k \left|{{p}}\left(x_i\right)-{p}\left(y_i\right)\right|^q \end{equation}
 The best $c$ above is denoted  $\pi^{Lip}_q(P)$.

Our goal now is to prove that for a homogeneous polynomial, both properties become equivalent.
Before stating  the result precisely, it is worth making a few comments concerning $q$-summability  for  polynomials.  Note that a function $f$ that satisfies    (\ref{eq: FarJ}) must be a Lipschitz map with    $\pi_q^L(f)$   a  Lipschitz constant for it.
 Since non-linear homogeneous polynomials  $P:X\rightarrow Y$ are not Lipschitz mappings,  they are never   Lipschitz $q$-summing maps with $M=X$ and $N=Y$  as in (\ref{eq: FarJ}). To overcome this difficulty, we will make use of Theorem \ref{thm: in symmetric terms}, which allows us to unambiguously use  the Lipschitz theory into the study of polynomials.

\begin{theorem}\label{thm: Lip=Lip} Let  $1\leq q < \infty$,  $X$ and  $Y$ be  real Banach spaces and let  $P:X \rightarrow Y$   be a  $d$-homogeneous polynomial. Then, $P$   is  a Lipschitz $q$-summing polynomial if and only if its associated Lipschitz mapping   $T_{P_|}:\mathbb{V}_X ^d\rightarrow Y$ satisfies (\ref{eq: FarJ}) for some  $C>0$.
In this case   $\pi^{Lip}_{q}(P)= \pi_q^L(T_{P_|})$.
\end{theorem}

The   case $d=1$  in Theorem \ref{thm: Lip=Lip} is the result for linear operators proved in \cite[Theorem 2]{FJ09}, since in this case
$\mathbb{V}^1_{X}=X$  and $P={T_{{P}}}={T_{{P}}}_{|_{\mathbb{V}^d_{X}}}\in{P}(^{1}{X},Y),$ as noted in Remark \ref{rmk: facotriz linear case}.

To prove the theorem, we will fix the notation and prove some results that we will need.
Given a finite dimensional subspace $E$ of $X$, $\gamma$ will   denote the  norm on $\otimes_{s}^dE$ induced by the inclusion in $\hat{\otimes}_{s,\pi}X$;   $\mathcal{P}_{\gamma}(^d E)$  will denote the  space of $d$-homogeneous polynomials,  with the norm induced  by  considering it isometric to $({\otimes}^d_{s,\gamma}E)^*$ and  $\mathbb{V}^d_{E, \gamma}:=(\mathbb{V}^d_E, d_\gamma)$.

\begin{lemma}\label{lem: pol Lip es local}   Let $P\in \mathcal{P}(^dX, Y)$ be  fixed. We will denote  $f_P$ the $\mathbb{V}$-operator associated to $P$ defined on $\mathbb{V}^d_{X}$ given by Theorem \ref{thm: in symmetric terms}.   Then
$$\pi^L_p(f_P)=\sup_{\{E\subset X\}}\{\pi^L_q({f_P}_{|\mathbb{V}^d_{E, \gamma}})\} \quad {\mbox{and}} \quad \pi^{Lip}_q(P)=\sup_{\{E\subset X\}}\{\pi^{Lip,\gamma}_q({P}_{|{E}})\}$$
where $\pi^{Lip,\gamma}_q({P}_{|{E}})$ denotes the best $c$ in  (\ref{eq: def pol -p-sumante}) when the supremum is taken over the unit ball of the scalar valued polynomials defined on $E$, with the  norm induced by $(\otimes_{s,\gamma}^d E)^*$.

\end{lemma}

\begin{proof}
   Both  constants $\pi^L_q(f_P)$ and $\pi^{Lip}_q(P)$  are computed considering (all) finite sets of $\mathbb{V}^d_{X}$ and $X$, respectively. The only aspect to be careful with is  the set of scalar valued functions considered for the supremum, which are  the metric spaces $B_{{\mathbb{V}_{X}^d}^{\#}}$ and $B_{\mathcal{P}(^{ d} {X})}$, respectively.  In each  case, in order to preserve the  metric when considering a subspace $E$ instead of the whole space $X$,    it is necessary to  consider the   norm  $\gamma$. In this way  $\mathbb{V}^d_{E, \gamma}$ is an isometric subspace of $\mathbb{V}^d_{X}$.   The equalities  follows in both cases, extending the respective mappings isometrically, by McShane theorem and by Hanh-Banach theorem, respectively.

\end{proof}

For now on, we will fix a finite dimensional subspace $E\subset X$.

\begin{proposition}\label{prop: Lip-Lip}  Let       $f\in   {Lip}_0(\mathbb{V}^d_{{E,\gamma}}, Y)$ be a   Lipschitz $q$-summing mapping with constant $  \pi^L_q(f)$. For each $w_0\in \hat{\otimes}_{s,\gamma}^dE$ there exists a probability measure $\mu_{w_0} \in \mathcal{M}(B_{( \hat{\otimes}_{s,\gamma}^dE)^{\#}})$  such that    for any
$u:=x\otimes\cdots\otimes x, v=y\otimes\cdots\otimes y $ in  $\mathbb{V}^d_{{E, \gamma}}$
\begin{equation}\label{eq: Pietsch mu2}
 \|f(u)-f(v)\|^q \leq \pi^L_q(f)^q  \int_{B_{(\hat{\otimes}_{s,\gamma}^d E)^{\#}}}|\zeta(u+w_0)-\zeta(v+w_0)|^qd\mu_{w_0}(\zeta).
 \end{equation}

\end{proposition}
\begin{proof}

    If    $f\in   {Lip}_0(\mathbb{V}^d_{{E,\gamma}}, Y)$ is  a Lipschitz $q$-summing mapping, then  by \cite[Theorem 1]{FJ09},  there exists a   probability  measure     $\mu$  defined  on $ B_{(\mathbb{V}^d_{{E,\gamma}})^{\#}}$ such that  a domination inequality of the form   (\ref{eq: Pietsch mu2}) holds,    where  the integral is computed on $M^{\#}:=B_{(\mathbb{V}^d_{{E,\gamma}})^{\#}}$ with respect to $\mu$, and $w_0$ does not appear (in our context we can consider it as $0$).   Now, consider  the continuous restriction mapping
\[\begin{array}{ccc}
   \varphi:  B_{(\hat{\otimes}_{s,\gamma}^d E)^{\#}} & \rightarrow & B_{({\mathbb{V}^d_{{E,\gamma}}})^{\#}} \\
    \zeta  & \mapsto & \zeta_{|_{\mathbb{V}^d_{{E,\gamma}}}}.
  \end{array}
\]
By
McShane's theorem, $\varphi$ is surjective. For each pair
$u=x\otimes\cdots \otimes x, v=y\otimes\cdots \otimes y  $ in  ${\mathbb{V}^d_{{E,\gamma}}}$, define   $h_{u,v}(\eta):=\eta(u)-\eta(v) \in \mathcal{C}(B_{({\mathbb{V}^d_{{E,\gamma}}})^{\#}})$,
$Z:=\{h_{u,v}; \,u,v\in {\mathbb{V}^d_{{E,\gamma}}}\}$,  $F:Z\rightarrow Y$, $F(h_{u,v}):=f(u)-f(v)$ and $C:= \pi^L_q(f)$. We have that  $\varphi$, $Z$,  $F$ and $C$  satisfy the hypotheses  of    \cite[Theorem 2.3]{LipLip23}
with $K_1= B_{(\otimes^d_{s,\gamma}E)^{\#}}$, $K_2= B_{(\mathbb{V}_{E,\gamma}^d)^{\#}}$ and $\rho_2=\mu$.
Since $\varphi$ is surjective, this theorem ensures te existence of  a probability  measure $\tilde{\mu}$ on  $B_{(\otimes^d_{s,\gamma}E)^{\#}}$ satisfying  that  for $u,v \in {\mathbb{V}^d_{E,\gamma}} $   $$
\|f(u)-f(v)\|\leq \pi^L_q(f)^q \int_{B_{(\hat{\otimes}_{s,\gamma}^d E)^{\#}}}|\zeta(u)-\zeta(v)|^qd\tilde{\mu}(\zeta).$$
To finish the proof,   let  $w_0\in {\otimes}^d_{s,\gamma}{{E}}$ and consider
    the bijective isometry   $
  \varphi_1: B_{({\otimes}^d_{s,\gamma}{{E}})^{\#}}\rightarrow B_{({\otimes}^d_{s,\gamma}{{E}})^{\#}} $ defined as $\varphi_1(\zeta)(w):=\zeta(w+w_0)-\zeta(w_0)$.
  We  use   \cite[Theorem 2.3]{LipLip23} to deduce the existence of a probability measure
 $\mu_{w_0}\in \mathcal{M}((\hat{\otimes}^d_{s,\gamma} E)^{\#})$ such that for every $u,v \in \mathbb{V}^d_{{E,\gamma}}$
   \begin{equation*}\label{eq: Pietsch FJ 3}
  \|f(u)-f(v)\|^q \leq \pi^L_q(f)^q  \int_{B_{(\otimes^d_{s,\gamma}E)^{\#}}}  |\zeta(u+w_0)-\zeta(v+w_0)|^qd{\mu}_{w_0}(\zeta)
 \end{equation*}
 which is (\ref{eq: Pietsch mu2}).

\end{proof}

Given a probability measure  $\mu\in \mathcal{M}(B_{({\otimes}^d_{s,\gamma} E)^{\#}})$ and a fixed  $w_0\in \otimes^d_{s,\gamma} E$, we will denote  $\alpha_{w_0}:\widehat{\otimes}^d_{s,\gamma}  E\rightarrow  L_{\infty}(\mu)$ the mapping  given by the relation  $\alpha_{w_0}(w):=\alpha({w+w_0})-\alpha(w_0)$, where
$\alpha(w)$ is   the image of $w$ under  the natural isometric embedding  of $\otimes^d_{s,\gamma} E $ into $\mathcal{C}(B_{({\otimes}^d_{s,\gamma} E)^{\#}})$ composed with the natural inclusion  into $L_{\infty}(\mu)$, $  w\stackrel{\alpha}{\mapsto} [\delta_w]$.
The  natural inclusion from  $L_{\infty}(\mu)$ into $L_{q}(\mu)$ will be denoted  $i_{\infty,q}$.

\begin{lemma}\label{lem: Lip-factrz} Let $E\subset X $ be a finite dimensional space and let   $f_P\in   \mathcal{L}(\mathbb{V}^d_{E, \gamma}, Y)$  be a  $\mathbb{V}$-operator which is a  Lipschitz $q$-summing map,  that is, it satisfies  (\ref{eq: FarJ}). Let   $w_0\in {\otimes}^d_{s,\gamma} E$ be fixed. Then, for  every $\epsilon >0$ there exist   a subspace $Y_E\subset Y$,  a linear embedding $J:Y_{E}\hookrightarrow \ell_{\infty}^m$   with $\|J\|=1, \|J^{-1}\|\leq 1+\epsilon$
a probability measure $\mu_{w_0}\in \mathcal{M}(B_{({\otimes}^d_{s,\gamma} {E})^{\#}})$
and a Lipschitz mapping $\tilde{\beta}:L_{q}(\mu_{w_0})\rightarrow \ell_{\infty}^m$ such that $Lip({\widetilde{\beta}})\leq \pi^L_q({f_{P}}_|)$ and  the following diagram commutes

\hspace{1cm}\xymatrix{
\mathbb{V}^d_{\gamma} \ar[r]^ {f_P}\ar[d]^{{{\alpha_{w_0}}_{|_{\mathbb{V}^d_{\gamma}}}}} &{Y_E}\ar[r]^{J}& \ell_{\infty}^m  \\
L_{\infty}(\mu_{w_0}) \ar[rr]^ {i_{\infty,q}}&&  L_q(\mu_{w_0}) \ar[u]^{\widetilde{\beta}}. &
}
\end{lemma}
\begin{proof}

 {{ Let ${Y_E}$ be the  finite dimensional subspace ${T_{{P}}}({\otimes}^d_{s,\gamma}{{E}})$ of $Y$. For each $\epsilon >0$, let  $J:Y_E\hookrightarrow \ell_{\infty}^m$ be a linear embedding with $\|J\|=1, \|J^{-1}\|\leq 1+\epsilon$.
By Proposition \ref{prop: Lip-Lip},  there exists a probability measure   $\mu_{w_0} \in \mathcal{M}(B_{({\otimes}^d_{s,\gamma} E)^{\#}})$ such inequality (\ref{eq: Pietsch mu2}) holds. This fact allows us  to define ${\beta}( (i_{\infty,q}{\alpha_{w_0}})(u) :=Jf_P(u)$ for every $u\in \mathbb{V}^d_{E,\gamma}$.  Then,  we have the  factorization $Jf_P={\beta} i_{\infty,q} {\alpha_{w_0}}_{|_{\mathbb{V}^d_{\gamma}}}$  with $Lip({\alpha_{w_0}})\leq 1$, $Lip({\beta})\leq \pi^L_q(f_P)$. By the injectivity of   $\ell_{\infty}^m$ we can extend   ${\beta}$ to $L_q(\mu_{x_0})$  preserving its norm. This extension is $\tilde{\beta}$.
}}

\end{proof}

 \begin{proof}[Proof of Theorem \ref{thm: Lip=Lip}]
By Proposition  \ref{prop: V-operators in Lipschitz}  with $Y=\mathbb{R}$, we know that   $ {\mathcal{L}\left({\mathbb{V}^d_{X}} \right)}$ and   ${\mathcal{P}\left(^{d}{X}\right)}$ are isometric spaces. Using   the isometric inclusion   $ \varphi :B_{\mathcal{L}\left({\mathbb{V}^d_{X}} \right)} \hookrightarrow B_{({\mathbb{V}^d_{X}})^{\#}}$, we derive  that whenever $P$ satisfies (\ref{eq: def pol -p-sumante}) for some $c>0$,  its associated $\mathbb{V}$-operator  $f_P$  satisfies (\ref{eq: FarJ})  for the metric space ${\mathbb{V}^d_{X}}$ with the same  constant.

  To  prove the reverse implication, we fix a finite dimensional subspace $E\subset X$. As before,   let $\gamma$  be  the  reasonable cross-norm on $\otimes^d_s E$ induced by the inclusion in $\widehat{\otimes}^d_{s,\pi}X$.  Thus,  the assumption is  that  $f_P:\mathbb{V}^d_{\gamma, E}\rightarrow \mathbb{R}$ satisfies (\ref{eq: FarJ}) with constant $\pi^L_q(f_{P_{|\mathbb{V}_E}})$.

  \noindent
{\underline{Claim:}}
Given  a probability measure $\mu\in \mathcal{M}(B_{({\otimes}_{s,\gamma}^d E)^{\#}})$  let   $\alpha:\widehat{\otimes}^d_{s,\gamma}  E\rightarrow  L_{\infty}(\mu)$ be, as before,  the natural isometric embedding into $\mathcal{C}(B_{({\otimes}_{s,\gamma}^d E)^{\#}})$ composed with the natural inclusion into $L_{\infty}(\mu)$, and let $i_{\infty,q}$ be the natural inclusion from $L_{\infty}(\mu)$ into $L_{q}(\mu)$. Then, for a.e.   $w_0\in \widehat{\otimes}^d_{s,\gamma}  E$,  the function $\alpha_{w_0}(w):=\alpha(w+w_0)-\alpha(w_0)$  is  $weak^*$-derivable  at $0$ and $i_{\infty,q}\circ \alpha_{w_0}$ is derivable at $0$. It holds that $\alpha_{w_0}(0)=0$ and $\|D_0^{w^*}(\alpha_{w_0})\|\leq Lip (\alpha). $

The  justification of the claim can  be derived using propositions 4.3 and 6.41 in  \cite{BenyLind}, as    done  in \cite[Theorem 2]{FJ09} and \cite[Theorem 2.1]{LipLip23}.

Let $w_0\in \widehat{\otimes}^d_{s,\gamma}  E$ be chosen as in the claim.  Since $f_{P_{|\mathbb{V}_E}}$ satisfies (\ref{eq: FarJ}), it satisfies a  factorization as in  Lemma  \ref{lem: Lip-factrz} with such $w_0$. Using the notation in the lemma,
let ${\beta}_n(w):=n\widetilde{\beta}(\frac{w}{n})$.  Then   $Lip({\beta}_n)=  Lip(\widetilde{\beta})$.

Giving  that $Jf_P$ is homogeneous, for   each $u:=x\otimes \cdots \otimes x  \in \mathbb{V}^d_{E,\gamma}$ it holds that
   \begin{multline*}
 \|J\circ f_P(u)-{\beta}_ni_{\infty,q}D_0^{w^*}\alpha_{{w_0}}(u) \|=\|n\widetilde{\beta}i_{\infty,q}\alpha_{{w_0}}(\frac{u}{n})-{\beta}_ni_{\infty,q}D_0^{w^*}\alpha_{{w_0}}(u)  \|   \\
= \|{\beta}_nni_{\infty,q}\alpha_{{w_0}}(\frac{u}{n})-{\beta}_ni_{\infty,q}D_0^{w^*}\alpha_{{w_0}}(u)  \| \leq
    Lip(\beta)\|ni_{\infty,q}\alpha_{w_0}(\frac{u}{n})-D_0(i_{\infty,q}\alpha_{w_0}({u}))\|\xrightarrow[n \to \infty]{} 0.
 \end{multline*}

 Since   $\{{\beta}_n\}$ is a bounded set in   $Lip_0(L_q(\mu_{w_0}), \ell_{\infty}^m)$, it  has   a cluster point $\beta_0$ in it.    Let $f_D:=D_0^{w^*}({\alpha_{{w_0}}})_{|_{\mathbb{V}\gamma}}$ and $P_D$  be  the $\mathbb{V}$-operator and the polynomial  mapping
 associated to $D_0^{w^*}({\alpha_{w_0}})$,   respectively.

  Then we have   the factorization  of the polynomial $J P=\beta_0 i_{\infty,q} P_{D}$ with $\|P_{D}\|=\|{D_0^{w^*}({\alpha_{{w_0}}})}\|_{{\widehat{\otimes}^d_{s,\pi} E\rightarrow L_{\infty}}} \leq \|{D_0^{w^*}({\alpha_{{w_0}}})}\|_{{{\otimes}^d_{s,\gamma} E\rightarrow L_{\infty}}}\leq Lip(\alpha_{w_0})$.   To finish  let $k\in\mathbb{N}$, $i=1,\ldots,k$ and  $u_i:=x_i\otimes\cdots\otimes x_i, v_i:=y_i\otimes\cdots\otimes y_i\in\mathbb{V}_{\gamma,E}$. Using that  the
$q$-summing norm  $\Pi_q(i_{\infty,q})$  of the linear inclusion is one,  we have that for
 $u_i, v_i \in \mathbb{V}_{\gamma,E}$
 \begin{multline*}
\sum_{i=1}^k \left\|Jf_P(u_i)-Jf_P(v_i)\right\|^q =
  \sum_{i=1}^k \left\|\beta_0i_{\infty,q}f_{D}(u_i)-\beta_0i_{\infty,q}f_{D}(v_i)\right\|^q
  \\\leq {Lip(\beta_0)}^q
  {\Pi_q(i_{\infty,q})}^q \sup_{ \varphi\in B_{L_{\infty}^*}} \sum_{i=1}^k \left|\varphi {D_0^{w^*}(\alpha_{w_0})}(u_i)-\varphi D_0^{w^*}(\alpha_{w_0})(v_i)\right|^q
\\ \leq    Lip(\beta_0)^q \|D_0^{w^*}(\alpha_{w_0})\|^q \sup_{ \zeta \in  B_{({\otimes}^d_{s,\gamma}E)^*}} \sum_{i=1}^k \left|\zeta(u_i)-\zeta(v_i)\right|^q \\
 \leq Lip(\beta_0)^q Lip(\alpha_{w_0})^q \sup_{ p \in B_{\mathcal{P}_{\gamma}(^dE)}} \sum_{i=1}^k \left|{{p}}\left(u_i\right)-{p}\left(v_i\right)\right|^q.
\end{multline*}
 Then $\pi_q^{Lip,\gamma}({{P}_{|_{{E}}}})\leq  Lip(\alpha_{w_0}) Lip({\beta_0}) $. Since this is true for an arbitrary  finite dimensional space $E\subset X$,  by Lemma \ref{lem: pol Lip es local},  we have that   $\pi_q^{Lip}({f_P})\leq  Lip(\alpha_{w_0}) Lip( {{\beta}_0}) $ and, consequently,  $\pi_q^{Lip}(P)\leq \pi_q^L(f_P)$.

\end{proof}

\end{document}